\documentclass[11pt]{article}

\usepackage{times}
\usepackage{amsthm}
\usepackage{amsmath}
\usepackage{amssymb}
\usepackage{mathrsfs}
\usepackage{pb-diagram}
\usepackage[all]{xy}
\usepackage{color}
\usepackage{comment}

\usepackage{hyperref}

\renewcommand\eqref[1]{(\ref{#1})} 

\def\A{{\mathcal A}}

\def\O{\Omega}
\def\th{\theta}

\def\N{\mathbb{N}}

\def\H{\mathcal H}

\def\R{\mathbb{R}}

\def\e{{\sf e}}
\def\ep{{\epsilon}}
\def\g{{\mathfrak g}}

\def\im{\mathop{\mathsf{Im}}\nolimits}
\def\re{\mathop{\mathsf{Re}}\nolimits}

\def\({\left(}
\def\[{\left[}
\def\){\right)}
\def\]{\right]}

\def\si{\sigma}

\def\a{\mathfrak{a}}
\def\d{\mathrm{d}}

\def\G{{\sf G}}

\def\p{\parallel}
\def\<{\langle}
\def\>{\rangle}

\def\a{{\sf a}}

\newtheorem{Theorem}{Theorem}[section]
\newtheorem{Remark}[Theorem]{Remark}
\newtheorem{Lemma}[Theorem]{Lemma}

\newtheorem{Corollary}[Theorem]{Corollary}
\newtheorem{Proposition}[Theorem]{Proposition}
\newtheorem{Definition}[Theorem]{Definition}
\newtheorem{Definition-Proposition}[Theorem]{Definition-Proposition}
\newtheorem{Example}[Theorem]{Example}

\numberwithin{equation}{section}

\begin{document}


\title{Resolvent Estimates and Smoothing for\\ Homogeneous Operators on Graded Lie Groups}

\date{\today}

\bigskip
\bigskip
\bigskip
\author{Marius M\u antoiu 
}
\date{\small}
\maketitle \vspace{-1cm}


\bigskip
\bigskip
\begin{abstract}
By using commutator methods, we show uniform resolvent estimates and obtain globally smooth operators for self-adjoint injective homogeneous operators $H$ on graded groups, including Rockland operators, sublaplacians and many others. Left or right invariance is not required. Typically the globally smooth operator has the form $T=V|H|^{1/2}$, where $V$ only depends on the homogeneous structure of the group through Sobolev spaces, the homogeneous dimension and the minimal and maximal dilation weights. For stratified groups improvements are obtained, by using a Hardy-type inequality. Some of the results involve refined estimates in terms of real interpolation spaces and are valid in an abstract setting. Even for the commutative group $\R^N$ some new classes of operators are treated.
\end{abstract}

\bigskip
\bigskip
\bigskip
\bigskip
\bigskip
{\bf Address}: Departamento de Matem\'aticas, Universidad de Chile, 

\smallskip
\quad\quad\quad\quad\ Las Palmeras 3425, Casilla 653, Santiago, Chile.

\bigskip
{\bf e-mail:} mantoiu@uchile.cl

\bigskip
\bigskip
\medskip
{\bf Acknowledgements:} 

\medskip
The author has been supported by Proyecto Fondecyt 1160359 and the N\'ucleo Milenio de F\'isica Matem\'atica RC120002.

\newpage

\section{Introduction}\label{duci}

Let $\G$ be a stratified Lie group of dimension $N$ and homogeneous dimension $M$ and $\Delta$ a (positive) sublaplacian \cite{FS,Ri,FR} on $\G$\,, seen as an unbounded self-adjoint operator in the Hilbert space $L^2(\G)$\,. For $\alpha>0$ one defines the (fractional) power $\Delta^\alpha$ by the usual functional calculus; it is still an unbounded self-adjoint positive operator. In particular, we set $|D|:=\Delta^{1/2}$. We are going to study unbounded self-adjoint homogeneous operators on the Hilbert space $\H:=L^2(\G)$\,. This includes the operators $\Delta^\alpha$ or their generalizations, positive fractional powers of Rockland operators on graded Lie groups \cite{FR}. Actually any injective self-adjoint homogeneous operator on an arbitrary Hilbert space is covered by our method, but our examples belong to the setting of graded groups endowed with their well-known homogeneous structures.

\smallskip
We start by describing two of our results. In a stratified group there is an action of the multiplicative group $(\R_+,\cdot)$ by dilations, that we will specify below; let $\nu_N\in\{1,2,\dots\}$ be the largest exponent (weight) of this action. We choose an homogeneous qusi-norm $\,[\cdot]:\G\to\R_+$ and denote by $[x]$ the operator of multiplication by the function $x\to[x]$ in $L^2(\G)$\,. The object to study is {\it a self-adjoint operator $H$ in $L^2(\G)$ which is supposed injective and homogeneous of strictly positive order with respect to the mentioned dilations} (left or right invariance is not needed). The absolute value $|H|$ and the strongly continuous unitary group $\big\{e^{itH}\!\mid\!t\in\R\big\}$ generated by $H$ are defined via the functional calculus. Proposition \ref{giurgiu} says that for any $\th\in(1/2,1]$ there is a positive constant $C_{\th}$ such that for any $u\in L^2(\G)$ one has
\begin{equation}\label{magurele}
\int_\R\,\big\Vert\,(1+[x])^{-\th(\nu_N-\nu_1)}[x]^{-\th\nu_1}({\sf Id}+|D|)^{-\th(\nu_N-\nu_1)}|D|^{-\th\nu_1}|H|^{1/2}e^{itH}u\,\big\Vert^2 dt\le C_\th\!\p\!u\!\p^2.
\end{equation}
One also has to impose $M\ge 3$\,; but all the non-commutative stratified Lie groups satisfy this. The number $\nu_1=1$ is the smallest exponent of the action by dilations. In our proofs it appears because we are using the generator of dilations, which contains vector fields $X_j$ which are homogeneous of various orders belonging to $\{\nu_1=1,2,\dots,\nu_N\}$\,. It seemed to us that including it explicitly in \eqref{magurele} is also a hint to the way parameters connected to the group are involved: For example the function $\lambda\to\varphi_\th(\lambda)$ defining the $|D|$-depending factor decays as $\lambda^{-\th\nu_N}$ at infinity and is singular as $\lambda^{-\th\nu_1}$ close to the origin. Similarly for the space behavior, expressed by a function applied to the homogeneous quasi-norm.

\smallskip
In the larger setting of graded Lie groups, sublaplacians are replaced by the more general concept of Rockland operator. If one chooses such an operator $R_0$\,, positive and homogeneous of degree $q$\,, then $\mathbf D:=R_0^{1/q}$ plays the same role as the operator $|D|$\,; in particular it serves to define admissible norms on the homogeneous and non-homogeneous Sobolev spaces $\dot L^2_1(\G)$ and $L^2_1(\G)$\,, respectively. Then, if $H$ is a self-adjoint injective homogeneous operators on the graded group, we show in Proposition \ref{ziceni} and Corollary \ref{urzziceni} the estimate
\begin{equation}\label{mugurele}
\int_\R\,\big\Vert\,(1+[x])^{-\th\nu_N}({\sf Id}+\mathbf D)^{-\th\nu_N}|H|^{1/2}e^{itH}u\,\big\Vert^2 dt\le C'_\th\!\p\!u\!\p^2.
\end{equation}
The factors singular at the origin are no longer present. Note again that the unbounded operator $|H|^{1/2}$ is involved. As far as we know, these are the first results of this type for such operators on non-commutative Lie groups.

\smallskip
The abstract form of \eqref{magurele} and \eqref{mugurele} is given by

\begin{Definition}\label{sfantugheorghe}
Let $H$ be a self-adjoint operator in the Hilbert space $\H$ and $T$ a densely defined operator with domain ${\sf Dom}(T)$\,. One says that $T$ is {\rm (globally) $H$-smooth} if, for some constant $C\ge 0$\,, one has 
\begin{equation}\label{turgusecuiesc}
\int_\R\,\big\Vert Te^{itH}v\big\Vert^2 dt\le C\!\p\!v\!\p^2\,,\quad\forall\,v\in\H\,.
\end{equation}
\end{Definition}

Thus, under the stated requirements, \eqref{magurele} tells us that the operator
\begin{equation}\label{fundulea}
U_{\th}(H):=(1+[x])^{-\th(\nu_N-\nu_1)}[x]^{-\th\nu_1}({\sf Id}+|D|)^{-\th(\nu_N-\nu_1)}|D|^{-\th\nu_1}|H|^{1/2}
\end{equation} 
is $H$-smooth. Note that $\,U_{\th}(H)=V_{\th}|H|^{1/2}$, where 
$$
V_{\th}:=(1+[x])^{-\th(\nu_N-\nu_1)}[x]^{-\th\nu_1}({\sf Id}+|D|)^{-\th(\nu_N-\nu_1)}|D|^{-\th\nu_1}
$$ 
only depends on the group $\G$\,, through the parameters $\nu_N$ and $\nu_1$ (the maximal and the minimal homogeneity degrees), the homogeneous quasi-norm $[\cdot]$ and the Sobolev-admissible operator $|D|$\,. Similar facts are true for \eqref{mugurele}.

\smallskip
This concept of smoothness is due to T. Kato \cite{Ka}, see also \cite[XIII.7]{ReSi}, and has a lot of interesting applications. There are weaker local notions that we do not discuss. Estimation \eqref{turgusecuiesc} is non-trivial even for bounded operators $T$\,; note, for example, that ${\sf Id}$ is never $H$-smooth. An important issue is to get examples of unbounded $H$-smooth operators $T$ and then \eqref{turgusecuiesc} will be seen as a strong property of regularization of the evolution group of $H$ with respect to the domain of $T$. Even ignoring the exigent square integrability, if \eqref{turgusecuiesc} holds, then, for every $v\in\H$\,, the vector $e^{itH}v$ should belong to ${\sf Dom}(T)$ for almost all the values of $t\in\R$\,. So one aims at getting "large" $H$-smooth operators and hope that for vectors $v$\,, membership of $e^{itH}v$ to ${\sf Dom}(T)$ and the estimation \eqref{turgusecuiesc} happens to be a transparent and/or interesting properties.
Since usually the smooth operators one obtains are products between a multiplication operator and a function of the derivatives, the interpretation is in terms of spatial behaviour and/or smoothing (gain of derivatives). Our operator $H$ can be very general; but if $H=\Delta^\alpha$, for instance, one has in \eqref{magurele} a factor behaving as $|D|^{\alpha-\th\nu_N}$ at infinity "in the spectral repesentation of $|D|$"\,.

\smallskip
The theory of smooth operators, including plenty of types of operators $H$ and concrete examples $T$, has a long history that we cannot review here. Most of the results (but not all) are for $\G=\R^N$, eventually seen as an Abelian Lie group. Staying mainly in the realm of "constant coefficient differential operators", we mention references \cite{CS,KY,BAD,Wa,BAK,MP,Su,Do,Ho,Ho1,Chi,RSug,Chi1,RSug1,RSug2,RSug3}, but there are many other articles dedicated to this topic.

\smallskip
If $\G=\!\R^N$ is Abelian, then $M=N$, the dilations are the standard ones, $\nu_1=\nu_N=1$ and $\Delta:=-\partial_1^2-\dots-\partial^2_N$ is the usual (positive) Laplace operator. Instead of $[\cdot]$ one can take any of the usual norms $|\cdot|$ on $\R^n$. So, if $H$ is a self-adjoint injective homogeneous operator in $L^2(\R^N)$ and $N\ge 3$ and $\th\in(1/2,1]$\,, then $|x|^{-\th}\Delta^{-\th/2}|H|^{1/2}$ {\it is globally $H$-smooth by our Proposition \ref{giurgiu}}. We believe that for such general operators $H$ this result is new even in the commutative case: note that $H$ is not supposed to be invariant or, if it is, $H=b(D)$ needs no assumption of ellipticity or dispersiveness of the symbol $b$\,. We also mention Corollary \ref{lippia}, saying that $({\sf id}+|A|)^{-\th}|H|^{1/2}$ is $H$-smooth, where $H$ verifies the same conditions as before, $A$ is the generator of dilations in $\R^N$ and $\th>1/2$\,. Actually one can take $\th=1/2+0^+$ in a certain sense, made precise by real interpolation; see Corollary \ref{lipia}. And these results are in fact valid for every graded group. 

\smallskip
In particular $|x|^{-\th}\Delta^{(\alpha-\th)/2}$ is $\Delta^\alpha$-smooth if $\alpha>0$\,, $\th\in(1/2,1]$ and $N\ge 3$\,. Initially, Kato and Yajima have shown this result for $\alpha=1$ in \cite{KY}\,. We refer to \cite[Th.\;5.2]{RSug2} for a better result, containing the case $\Delta^\alpha$ for commutative $\G$\,. Clearly, there are many results for $\R^N$ that cannot be obtained from ours.
We emphasis the non-commutative case in this article, which seems to be completely new. An interesting question is if recent techniques \cite{RSug2,RSug3} using canonical transformations and comparison principles could be pushed to the graded Lie group framework.

\smallskip
Most of the proofs concerning smooth operators rely on (i) Fourier analysis and restriction to submanifolds or (ii) resolvent estimates. It seems difficult to apply Fourier analysis for our non-commutative group, so we deduce globally $H$-smooth operators from resolvent estimates which are uniform when we get close to the real axis (a {\it global Limiting Absorption Principle}). 

\smallskip
The Limiting Absorption Principle is proved in Section \ref{firtanunusus}. Globality requires techniques different but related to Mourre's commutator methods \cite{Mou,ABG}.  The model is our treatment in \cite{BKM}, directed towards spectral analysis (see also \cite{MP,Ho}). We give a full proof for the convenience of the reader and because the setting is different from that of \cite{BKM}. In particular we obtain refined estimates, involving real interpolation spaces $\A_{1/2,1}$ associated to dilations acting (as a non-unitary $C_0$-group) in a sort of homogeneous Sobolev space; see also \cite{ABG} in the different context of Mourre theory. The reader will notice that in this section one only relies on the homogeneity, the positivity and the injectivity of our operator $H$ acting in some abstract Hilbert space.

\smallskip
Still in an abstract framework, Section \ref{firtanunus} is dedicated to deducing global smooth operators ($L^2$-time-dependent estimates) from the resolvent estimates. This is first done in Theorem \ref{iernut}, in a stadard way, using the spaces involved in the Limiting Absorption Principle as it was proven in the previous section. As said above, these spaces $\A_{1/2,1}$ are defined by dilations acting in a suitable homogeneous Sobolev-type space. In Lemma \ref{silistea} is shown how to replace them by simpler real interpolation spaces $\mathscr A_{1/2,1}$ of type $(1/2,1)$ between the basic Hilbert space $\H$ and the domain of the generator of the unitary dilation group in $\H$\,. As a consequence, the $H$-smooth operators have the form $V|H|^{1/2}$ with $V$ an operator not depending on $H$\,. Corollary \ref{lipia} gives conditions on an operator to be $H$-smooth, in terms of a Littlewood-Paley realization or a resolvent description of the simpler interpolation space. Then, for the sake of simplicity, we replace the real interpolation spaces of order $(1/2,1)$ by a complex interpolation space of order $1/2+\varepsilon$\,; this is the origin of the parameter $\th>1/2$ appearing in the subsequent statements as well as in \eqref{magurele} and \eqref{mugurele}. 

\smallskip
Section \ref{fralticeni} recalls some basic facts about graded groups, their dilation structure and their Rockland operators. It is not always easy to decide when a formally given operator is admissible for our treatment, i.\,e. self-adjoint (on some convenient domain), injective and homogeneous. We indicate some classes of examples.

\smallskip
After particularization to the group context, the results are still not so easy to grasp, because they are stated in terms of the generator of dilations. So in Section \ref{falticeni} we weaken them and put into evidence $H$-smooth operators that, besides the factor $|H|^{1/2}$, contain products of a multiplication operator and a suitable function of $|D|$\,. In particular one gets the estimation \eqref{mugurele}.
The setting requires a quite deep result \cite{CJ} on interpolation of analytic families of operators. For operators $H$ which are, in addition, left invariant, an improvement of the $H$-smooth operators is subject of Corollary \ref{frumusani}. It involves a weighted integration over left translations of the function $(1+[x])^{-\th\nu_N}$.

\smallskip
In the final section  the Lie group is stratified, so sublaplacians (particular cases of positive homogeneous Rockland operator) are available. The main new fact is the existing of certain Hardy-type inequalities proven by Ciatti, Cowling and Ricci \cite{CCR}, that allows a sharper approach. In this way one gets the $H$-smoothness of the operators \eqref{fundulea}.

\section{Global resolvent estimates for homogeneous operators}\label{firtanunusus}

To a self-adjoint operator $H$ in a Hilbert space $\H$ we associate canonically a scale of Hilbert spaces $\big\{\mathcal G^{\si}\!\mid\!\si\in\R\big\}$\,, where $\mathcal G^\si$ is the domain of $|H|^{\si/2}$ (constructed via the functional calculus) with scalar product
\begin{equation}\label{pitesti}
\<v,w\>_{\mathcal G^\si}:=\<v,w\>+\big<|H|^{\si/2}v,|H|^{\si/2}w\big>\,.
\end{equation}
With this convention one has ${\sf Dom}(H)=\mathcal G^2$\,, $\H=\mathcal G^0$, and $\mathcal G^1$ is the form-domain of $H$\,. One could call the spaces $\mathcal G^\si$ {\it the abstract Sobolev spaces associated to the operator $H$}\,; we will be interested only in the interval $\si\in[-2,2]$\,. In the usual way, $\mathcal G^{-\si}$ can be identified with the topological anti-dual of $\mathcal G^\si$ and the operator $H$ can be seen as a symmetric element of $\mathbb B(\mathcal G^{1},\mathcal G^{-1})$\,. 

\smallskip
Let us assume now that the self-adjoint operator $H$ is also injective. We define {\it the homogeneous abstract Sobolev spaces associated to $H$}, denoted by $\,\dot{\mathcal G}^\si$, $\si\in\R$\,. For positive $\si$\,, the space $\dot{\mathcal G}^\si$ will be the completion of $\mathcal G^\si$ in the norm deduced from the scalar product
\begin{equation}\label{huedin}
\<v,w\>_{\dot{\mathcal G}^\si}:=\big<|H|^{\si/2}v,|H|^{\si/2} w\big>\,.
\end{equation}
For negative $\si$ one uses the same formula for the scalar product to complete the space $|H|^{-\si}\mathcal G^{-\si}\subset\mathcal G^{\si}$.
To simplify notations, we set $\mathcal G:=\mathcal G^1$ and $\mathcal G^*:=\mathcal G^{-1}$, as well as $\dot{\mathcal G}:=\dot{\mathcal G}^1$ and $\dot{\mathcal G}^*:=\dot{\mathcal G}^{-1}$\,. 

\smallskip
Assume now that we are given an unitary and strongly continuous group $W$ in $\H$ such that $W(t)\mathcal G^2\subset\mathcal G^2$ for every $t\in\R$\,. Then, by suitable restrictions or extensions (justified by duality and interpolation), it defines naturally $C_0$-groups in all the Hilbert spaces $\mathcal G^\si$ and $\dot{\mathcal G}^\si$ with $\si\in[-2,2]$\,. Generally, they are no longer unitary. We will be particularly interested in the one leaving in $\dot{\mathcal G}^*=\dot{\mathcal G}^{-1}$, denoted by the same symbol $W$ for simplicity. Its infinitesimal generator, also denoted by $A$\,, is a closed densely defined operator in the Hilbert space $\dot{\mathcal G}^*$. The domain $\A\equiv{\sf Dom}(A;\dot{\mathcal G}^*)$ is endowed with the graph scalar product
\begin{equation}\label{suceava}
\<v,w\>_\A:=\<v,w\>_{\dot{\mathcal G}^*}\!+\<Av,Aw\>_{\dot{\mathcal G}^*}
\end{equation}
and its corresponding graph norm given by
\begin{equation}\label{sibiu}
\p\!v\!\p_\A^2\,=\,\p\!v\!\p_{\dot{\mathcal G}^*}^2\!+\p\!Av\!\p_{\dot{\mathcal G}^*}^2=\,\big\Vert |H|^{-1/2}v\big\Vert^2\!+\big\Vert |H|^{-1/2}Av\big\Vert^2.
\end{equation}
One gets a pair of Hilbert spaces $(\dot{\mathcal G}^*\!,\A)$\,, with $\A\hookrightarrow\dot{\mathcal G}^*$ (dense continuous embedding), to which real or complex interpolation can be applied.

\smallskip
Real interpolation spaces are usually defined via the $K$-method \cite{BL} but, due to the concrete feature that $\A$ is the domain of the generator of the $C_0$-group $W$ in the Hilbert space $\dot{\mathcal G}^*$, we prefer to introduce them via another admissible norm, cf. \cite[Sect.\; 2.7]{ABG} for instance. Considering only the relevant case $(\th,p)=(1/2,1)$\,, we set
\begin{equation}\label{napoca}
\big(\dot{\mathcal G}^*\!,\A\big)_{\frac{1}{2},1}\equiv\A_{1/2,1}:=\Big\{v\in\dot{\mathcal G}^*\,\Big\vert\,\int_0^1\!\tau^{-\frac{1}{2}}\!\p\!W(\tau)v-v\!\p_{\dot{\mathcal G}^*}\!\!\frac{d\tau}{\tau}<\infty\Big\}\,,
\end{equation}
which is a Banach space with the norm ($\ep_0>0$)
\begin{equation}\label{apahida}
\p\!v\!\p_{\A_{1/2,1}}:=\,\p\!v\!\p_{\dot{\mathcal G}^*}\!+\int_0^{\ep_0}\!\tau^{-\frac{1}{2}}\!\p\!W(\tau)v-v\!\p_{\dot{\mathcal G}^*}\!\!\frac{d\tau}{\tau}\,.
\end{equation}

\begin{Remark}\label{husi}
{\rm As a rule, we will work in the sense of Banachisable (or Hilbertisable) spaces, i.e., even if a norm (or a scalar product) is used to define a Banach space, we admit without many comments equivalent norms. All the expressions \eqref{apahida}, for different values of $\ep_0>0$\,, define equivalent scalar products. In Theorem \ref{maine} and its proof we will use a number $\ep_0$ conveniently defined. 
}
\end{Remark}

\begin{Remark}\label{costinesti}
{\rm In the proof of Theorem \ref{maine} we will need another point of view on the interpolation space $\A_{1/2,1}$\,, connected to the trace method \cite{BL,ABG}\,; we collect here just the useful part of the information, see for instance \cite[Prop.\;2.3.3]{ABG}: 
A vector $u\in\big(\dot{\mathcal G}^*\!,\A\big)_{1/2,1}$ can be approximated by a smooth family $\big\{u_\ep\mid\ep\in(0,\ep_0)\big\}\subset\A$
with a $C^1$-dependence on $\ep$ in the weaker norm $\p\!\cdot\!\p_{\dot{\mathcal G}^*}$ such that
\begin{equation}\label{ludus}
\int_0^{\ep_0}\Big[\p\!u_\tau'\!\p_{\dot{\mathcal G}^{*}}\!+\p\!Au_\tau\!\p_{\dot{\mathcal G}^{*}}\!\!\Big]\,\frac{d\tau}{\sqrt\tau}\le c\p\!u\!\p_{\A_{1/2,1}}
\end{equation}
and
\begin{equation}\label{faurei}
\p\!u-u_\ep\!\p_{\dot{\mathcal G}^{*}}\le c\,\ep^{1/2}\!\p\!u\!\p_{\A_{1/2,1}}\,.
\end{equation}
The positive constant $c$ is independent of $u$ and $\ep$\,.
}
\end{Remark}

\begin{Definition}\label{mangalia}
Let $H$ be a self-adjoint operator in the Hilbert space $\H$\,, with domain $\mathcal G^2$\,. Let $\big\{W(t)=e^{itA}\mid t\in\R\big\}$ be an unitary strongly continuous group of operators in $\H$ with infinitesimal generator $A$\,. We say that $H$ {\rm is homogeneouos of order $\alpha\in\R$ with respect to $W$ (or with respect to $A$)} if
\begin{itemize}
\item
$\,W(t)\mathcal G^2\subset\mathcal G^2\,,\quad\forall\,t\in\R\,,$
\item
$\,W(-t)HW(t)=e^{\alpha t}H\,,\quad\forall\,t\in\R\,.$
\end{itemize}
\end{Definition}

A direct consequence is the commutator relation 
\begin{equation}\label{venus}
i[H,A]=i\alpha H.
\end{equation} 
One way to understand the left hand side is to say that it is a (suitable type of derivative) in $t=0$ of the function $t\mapsto W(-t)HW(t)$\,.
 For us, the useful (and consistent) way is to say that 
 \begin{equation}\label{saturn}
 \<i[H,A]u,v\>:=i\<Hu,Av\>-i\<Au,Hv\>
 \end{equation}
 if $u,v$ belong to the domain ${\sf Dom}\big(A;\mathcal G^2\big)$ of the infinitesimal generator of $W$ acting in $\mathcal G^2$. With a suitable reinterpretation of the duality $\<\cdot,\cdot\>$\,, this will be written for $u,v\in\A$\,. 

For convenience, we include a diagram describing the relations between various Banach spaces. The arrows are linear, continuous and dense embeddings. The star denotes the topological anti-dual (or a space isomorphic to it). All the spaces with no subscript are actually Hilbert.
\begin{equation*}\label{patrat}
\begin{diagram}
\node{}\node{\mathcal G^2} \arrow{e}\node{\mathcal G=\mathcal G^1} \arrow{e} \arrow{s}\node{\mathcal H} \arrow{e}\node{\mathcal G^*=\mathcal G^{-1}}\arrow{e}\node{\mathcal G^{-2}}\\ 
\node{\A^*}\node{\A_{1/2,1}^*} \arrow{w}\node{\dot{\mathcal G}}\arrow{w} \node{}\node{\dot{\mathcal G}^*}\arrow{n}\node{\A_{1/2,1}}\arrow{w}\node{\A}\arrow{w}
\end{diagram}
\end{equation*}

From the embedding $\A_{1/2,1}\hookrightarrow\mathcal G^*$ and the fact that, for $\im z\ne 0$\,, $(H-z)^{-1}$ makes sense as a bounded operator $:\mathcal G^*\to\mathcal G$\,, we see that the relation \eqref{bucuresti} below makes sense, the bracket $\<\cdot,\cdot\>$ meaning the natural duality between $\mathcal G$ and $\mathcal G^*$. We are going to use a single notation for various compatible duality brackets.

\begin{Theorem}\label{maine}
Let $H$ be an injective self-adjoint operator in the Hilbert space $\H$\,, homogeneous of order $\alpha>0$ with respect to the generator $A$ of the strongly continuous unitary group $W$\,. There exists $C>0$ such that for every $u\in\A_{1/2,1}$ one has
\begin{equation}\label{bucuresti}
\sup_{\lambda\in\R}\,\sup_{\mu>0}\big\vert\big\<(H-\lambda\mp i\mu)^{-1}u,u\big\>\big\vert\le C\!\p\!u\!\p_{\A_{1/2,1}}^2.
\end{equation}
\end{Theorem}

\begin{Remark}\label{colibasi}
{\rm Of course, such a result is impossible for $\lambda\in{\sf sp}(H)$ and general $u\in\H$\,. The need to get an estimate uniform in $\lambda$ makes Mourre's commutator method \cite{Mou,ABG} unsuited. Uniformity in $\lambda$ is critical for obtaining {\it globally} $H$-smooth operators in Section \ref{firtanunus}. 
}
\end{Remark} 

\begin{Remark}\label{filiasi}
{\rm The spectrum of a homogeneous operator $H$ can only be one of the sets $\{0\}$\,, $\R$\,, $[0,\infty)$ or $(-\infty,0]$\,. From the relation $\,W(-t)HW(t)=e^{\alpha t}H$ it follows that $\,{\sf sp}(H)={\sf sp}(e^{\alpha t}H)=e^{\alpha t}{\sf sp}(H)$\,. Since this is true for every $t\in\R$\,, the assertion follows easily. But we obtain more. A standard consequence of the Limiting Absorption Principle stated in Theorem \ref{maine} is the fact that, if $H$ is injective, {\it it has only purely absolutely continuous spectrum}.
}
\end{Remark}

We are now going to prove Theorem \ref{maine} by using a modification of {\it the method of differential inequalities}, invented by E. Mourre \cite{Mou} (see also \cite{ABG,BKM}). Some arguments will rely on positivity, and this requires some preparations. If, besides self-adjointness and homogeneity of $H$, the reader accepts {\it positivity} as an assumption, the proof below would be somehow easier: in the next constructions one could just take $H=H_>$ and some computations will be simpler. For the general case, let us decompose $H$ in a positive and a negative part, writting 
\begin{equation}\label{constanta}
H=H_>\!+H_<=HE_H(0,\infty)+HE_H(-\infty,0)\equiv HE_>\!+HE_<\,.
\end{equation}
Here $E_H$ is the spectral measure of $H$\,; since $H$ is assumed injective, one has $E_H(\{0\})=0$ and clearly $|H|=H_>\!-H_<$\,. Setting $\H_\pm:=E_\pm\H$ one gets an orthogonal decomposition $\H=\H_>\oplus\H_<$ which is invariant under the functional calculus of $H$. 

\smallskip
From homogeneity we deduce $e^{-itA}E_H(B)e^{itA}=E_H\big(e^{\alpha t}B\big)$ for every $t\in\R$ and every real Borel set $B$\,. Since the sets $(-\infty,0)$ and $(0,\infty)$ are invariant under dilations, we see that the spaces $\H_>$ and $\H_<$ are also invariant under the group generated by $A$\,.
In the proof of Theorem \eqref{maine} we are going to use the decomposition $A=A_>\!+A_<$ corresponding to the invariant direct sum $\H=\H_>\!\oplus\H_<$ and the operator ${\sf A}:=A_>\!-A_<$\,, as well as simple commutator computations
\begin{equation}\label{navodari}
\big[A_>\,,H_>\big]=iH_>\,,\quad\big[A_<\,,H_<\big]=iH_<\,,\quad\big[A_>\,,H_<\big]=0\,,\quad\big[A_<\,,H_>\big]=0\,,
\end{equation}
based on homogeneity and definitions.

\begin{proof}
1. For positive $\ep$ we set
\begin{equation*}\label{mamaia}
H_\ep^\mp:=e^{\mp i\alpha\epsilon}H_>+e^{\pm i\alpha\epsilon}H_<\,.
\end{equation*}
Let us show that for every $\mu>0$ and $\ep\in[0,\pi/\alpha]$ the two operators $\,H_\ep^\mp-\lambda\mp i\mu\,$ are isomorphisms $:\mathcal G^{2}\to\H$\,. We first check that they are bijections. One computes for $v\in\mathcal G^{2}$ 
\begin{equation}\label{oravita}
\begin{aligned}
\big\Vert\big(H_\ep^\mp-\lambda\mp i\mu\big)v\big\Vert^2&=\big\Vert\big(H_\ep^\mp-\lambda\big)v\big\Vert^2+\mu^2\!\p\!v\!\p^2+\,2\re\big\<\big(e^{\mp i\alpha\epsilon}H_>\!+e^{\pm i\alpha\epsilon}H_<\!-\lambda\big)v,\mp i\mu v\big\>\\
&\ge\mu^2\!\p\!v\!\p^2+\,2\re\big[e^{\mp i\alpha\epsilon}(\pm i\mu)\big]\big\<H_>v, v\big\>+2\re\big[e^{\pm i\alpha\epsilon}(\pm i\mu)\big]\big\<H_<v, v\big\>\big]\\
&=\mu^2\!\p\!v\!\p^2\mp\,2\mu\sin(\mp\alpha\ep)\big\<H_>v, v\big\>\mp\,2\mu\sin(\pm\alpha\ep)\big\<H_<v, v\big\>\\
&=\mu^2\!\p\!v\!\p^2\pm\,2\mu\sin(\pm\alpha\ep)\big[\big\<H_>v, v\big\>-\big\<H_<v, v\big\>\big]\\
&=\mu^2\!\p\!v\!\p^2+2\mu\sin(\alpha\ep)\big\<|H|v, v\big\>\ge\mu^2\!\p\!v\!\p^2,
\end{aligned}
\end{equation}
which proves injectivity. Since the two closed operators are adjoint to each other, their ranges are dense. By the computation \eqref{oravita}, their inverses (defined on the ranges, respectively, and seen as operators in $\H$) are bounded. But the inverse of a closed operator is also closed. Then, by the Closed Graph Theorem, their ranges are closed in $\H$\,, so they coincide with $\H$ and the isomorphism is proven.

\smallskip
2. From \eqref{oravita} one clearly gets 
\begin{equation}\label{pantelimon}
\big\Vert(H_\ep^\mp-\lambda\mp i\mu)^{-1}\big\Vert_{\mathbb B(\H)}\le 1/\mu\,.
\end{equation} 
Later on we will need a version of this inequality which reads
\begin{equation}\label{tulcea}
\big\Vert(H_\ep^\mp-\lambda\mp i\mu)^{-1}\big\Vert_{\mathbb B(\mathcal G^*\!,\mathcal G)}\le c(\lambda,\mu)\,,\quad{\rm uniformly\ in\ } \ep\ge 0\,.
\end{equation}
If we prove this with $\p\!\cdot\!\p_{\mathbb B(\mathcal G^*\!,\mathcal G)}$ replaced by $\p\!\cdot\!\p_{\mathbb B(\H,\mathcal G^2)}$\,, by duality and interpolation one gets \eqref{tulcea} easily. For this it is enough to control $\big\Vert H(H_\ep^\mp-\lambda\mp i\mu)^{-1}\big\Vert_{\mathbb B(\H)}$\,. Using the invariant decomposition $\H=\H_>\!\oplus\H_<$\,, this reduces to the control of $\big\Vert e^{\mp{i\alpha\ep}}H_>(e^{\mp{i\alpha\ep}}H_>-z^\mp{\sf Id}_>)^{-1}\big\Vert_{\mathbb B(\H_>)}$ and of a similar tem, with $>$ replaced by $<$\,. The identity $K(K-z)^{-1}={\sf Id}+z(K-z)^{-1}$ and \eqref{pantelimon} finishes the proof; we leave the details to the reader.

\smallskip
3. Starting from now, we are going to treat the sign $-$\,; the sign $+$ is similar. Let us define 
\begin{equation*}\label{iasi}
F_\epsilon\equiv F_\epsilon(\lambda,\mu):=\big(e^{-i\alpha\epsilon}H_>+e^{i\alpha\epsilon}H_<-\lambda-i\mu\big)^{-1}\in\mathbb B\big(\H,\mathcal G^{2})\,.
\end{equation*}
By transposition and interpolation, it can be seen as well-defined element of $\,\mathbb B\big(\mathcal G^{*},\mathcal G\big)\subset\mathbb B\big(\dot{\mathcal G}^{*}\!,\dot{\mathcal G}\big)$\,.
In the next computations, $\<\cdot,\cdot\>$ will denote the duality form between $\dot{\mathcal G}$ and $\dot{\mathcal G}^{*}$. Then, for a sequence $\{u_\epsilon\}_{\ep\in(0,\ep_0)}\subset\A$ as in Remark \ref{costinesti}, set 
\begin{equation*}\label{cluj}
f_\epsilon\equiv f_\epsilon(\lambda,\mu;u_\ep):=\big<F_\epsilon u_\ep,u_\ep\big>\,.
\end{equation*}
Using the basic rules of derivation and relations \eqref{navodari} one can write
$$
\begin{aligned}
\frac{d}{d\epsilon}f_\epsilon&=\big<F_\epsilon\big(i\alpha e^{-i\alpha\epsilon}H_>-i\alpha e^{i\alpha\epsilon}H_<\big)F_\epsilon u_\ep,u_\ep\big>+\big<F_\epsilon u'_\ep,u_\ep\big>+\big<F_\epsilon u_\ep,u'_\ep\big>\\
&=\big<F_\epsilon\big(e^{-i\alpha\epsilon}[A_>,H_>]-e^{i\alpha\epsilon}[A_<,H_<]\big)F_\epsilon u_\ep,u_\ep\big>+\big<F_\epsilon u'_\ep,u_\ep\big>+\big<F_\epsilon u_\ep,u'_\ep\big>\\
&=-\big<F_\epsilon[ e^{-i\alpha\epsilon}H_>\!+e^{i\alpha\epsilon}H_<\!-\lambda-i\mu,A_>\!-A_<]F_\epsilon u_\ep,u_\ep\big>+\big<u'_\ep,F_\epsilon^* u_\ep\big>+\big<F_\epsilon u_\ep,u'_\ep\big>\\
&=-\big<F_\epsilon u_\ep,{\sf A}u_\ep\big>+\big<{\sf A}u_\ep,F^*_\epsilon u_\ep\big>+\big<u'_\ep,F_\epsilon ^*u_\ep\big>+\big<F_\epsilon u_\ep,u'_\ep\big>\\
&=\big<F_\epsilon u_\ep,u'_\ep-{\sf A}u_\ep\big>+\big<u'_\ep+{\sf A}u_\ep,F^*_\epsilon u_\ep\big>\,.
\end{aligned}
$$
Note that $\p\!{\sf A}u_\ep\!\p_{\dot{\mathcal G}^*}=\,\p\!Au_\ep\!\p_{\dot{\mathcal G}^*}$ (use the orthogonal decomposition $\H=\H_>\oplus\H_<$\,, with respect to which both $A$ and ${\sf A}$ are diagonal). Since $\,|\<v,w\>|\le\,\p\!v\!\p_{\dot{\mathcal G}}\p\!w\!\p_{\dot{\mathcal G}^*}$ for $v\in\dot{\mathcal G}$ and $w\in\dot{\mathcal G}^*$, it follows that
\begin{equation}\label{craiova}
\big\vert f_\ep'\big\vert\le\Big(\p\!u_\ep'\!\p_{\dot{\mathcal G}^*}\!+\p\!Au_\ep\!\p_{\dot{\mathcal G}^*}\!\Big)\Big(\big\Vert F_\epsilon u_\ep\big\Vert_{\dot{\mathcal G}}+\big\Vert F_\epsilon^* u_\ep\big\Vert_{\dot{\mathcal G}}\Big)\,.
\end{equation}

4. We can write
\begin{equation*}\label{petrosani}
\begin{aligned}
\frac{1}{2i}\big(F_\epsilon-F_\epsilon^*\big)&=\frac{1}{2i}F_\epsilon\big([e^{i\alpha\epsilon}-e^{-i\alpha\epsilon}]H_>+[e^{-i\alpha\epsilon}-e^{i\alpha\epsilon}]H_<+2i\mu\big) F_\epsilon^*\\
&\ge\sin(\alpha\epsilon)F_\epsilon (H_>\!-H_<) F_\epsilon^*=\sin(\alpha\epsilon)F_\epsilon |H| F_\epsilon^*
\end{aligned}
\end{equation*}
so, using the fact that $\,\sin\tau/\tau\,$ is close to one if $\,\tau$ is small enough, we easily get
\begin{equation*}\label{slatina}
\big\Vert F^*_\epsilon u_\ep\big\Vert_{\dot{\mathcal G}}=\big<|H| F^*_\epsilon u_\ep,F^*_\epsilon u_\ep\big>^{1/2}\le\frac{1}{\sqrt{\sin(\alpha\epsilon)}}\;\big\vert\<F_\epsilon u_\ep,u_\ep\>\big\vert^{1/2}\le\frac{2}{\sqrt{\alpha\ep}}\;\big\vert\<F_\epsilon u_\ep,u_\ep\>\big\vert^{1/2},
\end{equation*}
assuming that $\ep\in(0,\ep_0)$ with $\ep_0>0$ small enough. Actually the argument works similarly for $F_\ep$ and leads to
\begin{equation}\label{caracal}
\big\Vert F_\epsilon u_\ep\big\Vert_{\dot{\mathcal G}}\le \;\frac{2}{\sqrt{\alpha\ep}}\big\vert\<F_\epsilon u_\ep,u_\ep\>\big\vert^{1/2},\quad\forall\,\ep\in(0,\ep_0)\,.
\end{equation}

5. Replacing \eqref{caracal} in \eqref{craiova} one gets the differential inequality
\begin{equation*}\label{calafatt}
\big\vert f_\epsilon'\big\vert\le\,\frac{4}{\sqrt\alpha}\,\frac{\p\!u_\ep'\!\p_{\dot{\mathcal G}^*}+\p\!Au_\ep\!\p_{\dot{\mathcal G}^*}}{\sqrt\ep}\,|f_\ep|^{1/2},\quad\forall\,\ep\in(0,\ep_0)\,.
\end{equation*}
The crucial fact is that the factor multiplying $|f_\ep|^{1/2}$ is integrable on $[0,\ep_0)$\,. A version of Gronwall's Lemma \cite[Th.5]{Dr} implies that the limit $f_0:=\lim_{\ep\to 0} f_\ep$ exists and satisfies the explicit inequality
\begin{equation*}\label{lupeni}
\big\vert f_0\big\vert\le\(\big\vert f_{\epsilon_0}\big\vert^{1/2}\!+\frac{2}{\sqrt\alpha}\int_{\ep}^{\ep_0} \frac{\p\!u_\tau'\!\p_{\dot{\mathcal G}^*}+\p\!Au_\tau\!\p_{\dot{\mathcal G}^*}}{\sqrt\tau}\,d\tau\)^2.
\end{equation*}
Applying \eqref{ludus}, this implies
\begin{equation}\label{petrila}
\big\vert f_0\big\vert\le\(\,\big\vert f_{\epsilon_0}\big\vert^{1/2}\!+c_1\p\!u\!\p_{\A_{\frac{1}{2},1}}\)^2.
\end{equation}

6. Now we estimate the first term in the r.h.s. of \eqref{petrila}. From \eqref{caracal} one gets immediately
\begin{equation*}\label{sadova}
\big\Vert F_\epsilon u_\ep\big\Vert_{\dot{\mathcal G}}\le\frac{4}{\alpha\ep}\p\!u_\ep\!\p_{\dot{\mathcal G}^*}.
\end{equation*}
On the other hand, by \eqref{faurei},
\begin{equation*}\label{arad}
\p\!u_\ep\!\p_{\dot{\mathcal G}^*}\le\,\p\!u_\ep-u\!\p_{\dot{\mathcal G}^*}\!\!+\!\p\!u\!\p_{\dot{\mathcal G}^*}\le c\,\ep^{1/2}\!\p\!u\!\p_{\A_{\frac{1}{2},1}}\!\!+\!\p\!u\!\p_{\A_{\frac{1}{2},1}}\!.
\end{equation*}
So
\begin{equation*}\label{ostravita}
\big\vert f_{\epsilon_0}\big\vert^{1/2}\le\big\Vert F_{\epsilon_0}\big\Vert_{\mathbb B(\dot{\mathcal G}^*,\dot{\mathcal G})}^{1/2}\p\!u_{\ep_0}\!\p_{\dot{\mathcal G}^*}\le c_3(\ep_0)\p\!u\!\p_{\A_{\frac{1}{2},1}}\,,
\end{equation*}
which replaced into \eqref{petrila} leads to
\begin{equation}\label{oradea}
\big\vert f_0\big\vert\le C\!\p\!u\!\p_{\A_{\frac{1}{2},1}}^2\,,
\end{equation}
uniformly in $\lambda\in\R$ and $\mu>0$\,.

\smallskip
7. We now compare $\,f_0=\lim_{\ep\to 0}\big<F_\epsilon u_\ep,u_\ep\big>$ with $\,\<F_0u,u\>=\big\<(H-\lambda-i\mu)^{-1}u,u\big\>$\,. One writes
\begin{equation*}\label{turnu}
\begin{aligned}
\big\vert\big<F_\epsilon u_\ep,u_\ep\big>-\big<F_0 u,u\big>\big\vert&\le\Vert F_{0}\big\Vert_{\mathbb B(\dot{\mathcal G}^*\!,\dot{\mathcal G})}\!\p\!u_\ep-u\!\p_{\dot{\mathcal G}^*}\!\!\big(\!\p\!u\!\p_{\dot{\mathcal G}^*}+\p\!u_\ep\!\p_{\dot{\mathcal G}^*}\!\!\big)\!+\Vert F_{\epsilon}-F_0\big\Vert_{\mathbb B(\dot{\mathcal G}^*\!,\dot{\mathcal G})}\!\p\!u_\ep\!\p_{\dot{\mathcal G}^*}^2\\
&\le\Vert F_{0}\big\Vert_{\mathbb B(\mathcal G^*\!,\mathcal G)}\!\p\!u_\ep-u\!\p_{\dot{\mathcal G}^*}\!\!\big(\!\p\!u\!\p_{\dot{\mathcal G}^*}+\p\!u_\ep\!\p_{\dot{\mathcal G}^*}\!\big)\!+\Vert F_{\epsilon}-F_0\big\Vert_{\mathbb B(\mathcal G^*\!,\mathcal G)}\!\p\!u_\ep\!\p_{\dot{\mathcal G}^*}^2\!.
\end{aligned}
\end{equation*}
From \eqref{tulcea}, that we write $\Vert F_{\epsilon}\big\Vert_{\mathbb B(\mathcal G^*,\mathcal G)}\!\le c(\lambda,\mu)$\,, one gets for every $\epsilon\ge 0$
\begin{equation*}\label{hunedoara}
\begin{aligned}
\Vert F_{\epsilon}-F_0\big\Vert_{\mathbb B(\mathcal G^*,\mathcal G)}&\le\Vert F_{\epsilon}\big\Vert_{\mathbb B(\mathcal G^*,\mathcal G)}\big\Vert e^{-i\alpha\ep}H_>\!+e^{i\alpha\ep}H_<\!-H\big\Vert_{\mathbb B(\mathcal G,\mathcal G^*)}\Vert F_{0}\big\Vert_{\mathbb B(\mathcal G^*,\mathcal G)}\\
&\le c(\lambda,\mu)^2\big(|e^{-i\alpha\ep}-1|\p\! H_>\!\p_{\mathbb B(\mathcal G,\mathcal G^*)}+|e^{i\alpha\ep}-1|\p\! H_<\!\p_{\mathbb B(\mathcal G,\mathcal G^*)}\big)\,,
\end{aligned}
\end{equation*}
which converges to $0$ when $\ep\to\infty$\,.
Also using \eqref{faurei}, we obtain finally (non-uniformly in $\lambda,\mu$) 
\begin{equation*}\label{iulia}
f_0:=\lim_{\ep\to 0}\big<F_\epsilon(\lambda,\mu) u_\ep,u_\ep\big>=\big<F_0(\lambda,\mu) u,u\big>=\big\<(H-\lambda- i\mu)^{-1}u,u\big\>\,.
\end{equation*}
Then, since \eqref{oradea} holds uniformly in $\lambda$ and $\mu$\,, one gets
\begin{equation*}\label{satumare}
\big\vert\big\<(H-\lambda- i\mu)^{-1}u,u\big\>\big\vert\le C\!\p\!u\!\p_{\A_{\frac{1}{2},1}}^2\,,\quad\forall\,\lambda\in\R\,,\,\mu>0\,,\,u\in\A_{\frac{1}{2},1}\equiv(\dot{\mathcal G}^*\!,\A)_{\frac{1}{2},1}\,.
\end{equation*}

\end{proof}

\section{Globally smooth operators under homogeneity}\label{firtanunus}

Global smoothness has been defined in the Introduction. An equivalent definition of the $H$-smoothness of the operator $T$ is
\begin{equation*}\label{intorsatura}
\underset{\p u\p=1,\mu\ne 0}{\sup}\int_\R\,\big\Vert T(H-\lambda\mp i\mu)^{-1}u\big\Vert^2 d\lambda<\infty\,,
\end{equation*}
where it is assumed that for every $u\in\H$ and for every $\mu>0$\,, one has $(H-\lambda\mp i\mu)^{-1}u\in{\sf Dom}(T)$ for almost every $\lambda\in\R$\,.

\smallskip
Comming back to our concrete setting, it is useful to recall the dense continuous embeddings
\begin{equation}\label{medias}
\A\hookrightarrow\A_{1/2,1}\hookrightarrow\dot{\mathcal G}^*\hookrightarrow\mathcal G^*\equiv\mathcal G^{-1}\hookrightarrow\mathcal G^{-2}.
\end{equation}
It is known \cite[Th.\! XIII.25]{ReSi} that if $T$ is $H$-smooth then $\mathcal G^{2}={\sf Dom}(H)\hookrightarrow{\sf Dom}(T)$\,, hence $T$  can be seen as a bounded linear operator $T:\mathcal G^{2}\to\H$\,. Therefore, its adjoint (after extension) can also be regarded as a bounded linear operator $T^*:\H\to\mathcal G^{-2}$.
Then, taking \eqref{medias} into account, it makes sense to require that $T^*\H\subset\A_{1/2,1}$ and that, seen as a linear map $\,T^*\!:\H\to\A_{1/2,1}$\,, it is an element of $\mathbb B\big(\H,\A_{1/2,1}\big)$\,.

\begin{Theorem}\label{iernut}
Let $T$ be a densely defined operator in $\H$ such its adjoint $\,T^*\!$ extends to an element of $\,\mathbb B\big(\H,\A_{1/2,1}\big)$\,. Then $T$ is globally $H$-smooth.
\end{Theorem}

\begin{proof}
By \cite[Th.\! 3.25]{ReSi}, the operator $T$ is $H$-smooth if and only if
\begin{equation}\label{targumures}
\sup\Big\{\big\vert\big<\big[(H-z)^{-1}-(H-\overline z)^{-1}\big]T^*v,T^*v\big>\big\vert\;\big\vert\;v\in{\sf Dom}(T^*),\,\p\!v\!\p\,=1,\,z\notin\R\Big\}<\infty\,.
\end{equation}
In Theorem \ref{maine} we obtained the uniform estimate
\begin{equation*}\label{crevedia}
\sup_{z\notin\R}\big\vert\big\<(H-z)^{-1}u,u\big\>\big\vert\le C\!\p\!u\!\p_{\A_{\frac{1}{2},1}}^2.
\end{equation*}
Setting $u=T^*v$ with $v\in\H$\,, this reads
\begin{equation}\label{tarlungeni}
\sup_{z\notin\R}\big\vert\big\<(H-z)^{-1}T^*v,T^*v\big\>\big\vert\le C\!\p\!T^*v\!\p_{\A_{\frac{1}{2},1}}^2\le C(T^*)\!\p\!v\!\p^2,
\end{equation}
which clearly implies \eqref{targumures}. Here the interpretation of $T^*$ is the one indicated before the Corollary; but if boundeness of $T^*$ with respect to the relevant topologies is granted, in \eqref{tarlungeni} it is enough to take $v\in{\sf Dom}(T^*)$\,.
\end{proof}

\begin{Remark}\label{turceni}
{\rm For writing the condition of $H$-smoothness directly on $T$, it is useful to recall the identification between the anti-dual $\A_{1/2,1}^*$ of the space $\A_{1/2,1}$ and the interpolation space $\big(\A^*,\dot{\mathcal G}\big)_{1/2,\infty}$ associated to the couple of Hilbert spaces $\dot{\mathcal G}\cong(\dot{\mathcal G}^*)^*\!\hookrightarrow\A^*$. One of the admissible norms on $\big(\A^*,\dot{\mathcal G}\big)_{1/2,\infty}$ is
\begin{equation*}\label{simeria}
w\to\sup_{\tau\in(0,1]}\Big[\frac{1}{\sqrt\tau}\p\!W(\tau)w-w\!\p_{\A^*}\!\!\Big]\,.
\end{equation*}
Let us also set $\A_{1/2,\infty}^0\equiv\big(\A^*,\dot{\mathcal G}\big)^0_{1/2,\infty}$ for the closure of $\,\dot{\mathcal G}$ (or of $\,\mathcal G$) in $\big(\A^*,\dot{\mathcal G}\big)_{1/2,\infty}$\,. It is known \cite{BL} that 
$$
\Big[\big(\A^*,\dot{\mathcal G}\big)_{1/2,\infty}^0\Big]^*\cong\big(\dot{\mathcal G}^*\!,\A\big)_{1/2,1}=:\A_{1/2,1}\,.
$$
Then every operator $\,T\in\mathbb B\big(\A_{1/2,\infty}^0,\H\big)$ is globally $H$-smooth. }
\end{Remark}

Te space $\A_{1/2,1}$ is somehow intricate; it involves the homogeneous space $\dot{\mathcal G}^*$ and the generator $A$ of a $C_0$-group in $\dot{\mathcal G}^*$ which is not unitary and, very often, not even polynomially bounded. A useful simplification is brought by the following lemma:

\begin{Lemma}\label{silistea}
Let us denote by $\mathscr A$ the domain of $A=A^*$ in the Hilbert space $\H$ (with the graph norm) and by $\mathscr A_{s,p}=(\H,\mathscr A)_{s,p}$ the real interpolation spaces associated to the couple $(\H,\mathscr A)$\,, with $s\in(0,1)$ and $p\in[1,\infty]$\,.
\begin{enumerate}
\item[(i)]
The map $|H|^{1/2}\!:\H\to\dot{\mathcal G}^*$ restricts to a linear homeomorphism $\mathscr A_{s,p}\to\A_{s,p}$\,.
\item[(ii)]
$L\in\mathbb B\big(\H,\mathscr A_{s,p}\big)$ if and only if $\,|H|^{1/2}L\in\mathbb B\big(\H,\A_{s,p}\big)$\,.
\item[(iii)]
If $L\in\mathbb B\big(\H,\mathscr A_{1/2,1}\big)$\,, then $T\!:=L^*|H|^{1/2}$ is $H$-smooth.
\end{enumerate}
\end{Lemma} 

\begin{proof}
Note first that $|H|^{1/2}:\H\to\dot{\mathcal G}^*$ is a (unitary) isomorphism. By the homogeneity relation
$$
A|H|^{\pm\frac{1}{2}}=|H|^{\pm\frac{1}{2}}A\pm(i\alpha/2)|H|^{\pm\frac{1}{2}}
$$
and the explicit graph norms, one shows immediately that 
$$
{\sf Dom}(A;\H)=\mathscr A\overset{|H|^{\frac{1}{2}}}{\longrightarrow}\A={\sf Dom}(A;\dot{\mathcal G}^*)
$$
is a linear homeomorphism. Then the assertion (i) follows from real interpolation and (ii) is a direct consequence.

\smallskip
The point (iii) follows from (ii) (with $(s,p)=(1/2,1)$) and Theorem \ref{iernut}.
\end{proof}

So we are reduced to find interesting elements of $\mathbb B\big(\H,\mathscr A_{1/2,1}\big)$\,. Since the group generated by $A$ in $\H$ is unitary, we have two other characterizations of the space $\mathscr A_{1/2,1}$\,; the first one is of Littlewood-Paley type and the second involves the resolvent of the self-adjoint operator $A$\,. 

\begin{enumerate}
\item[(i)]
Let $0<a<b<\infty$\,, let $\eta\in C^\infty_0(\R)$ such that $\eta(\lambda)>0$ if $|\lambda|\in(a,b)$ and $\eta(\lambda)=0$ if $|\lambda|\notin(a,b)$ and let $\zeta\in C^\infty_0(\R)$ with $\,\zeta(\lambda)=1$ if $\lambda\in[-b,b]$\,. Then
\begin{equation}\label{nuci}
v\to\,\p\!\zeta(A)v\!\p+\int_0^1\!\p\!\eta(\lambda A)v\!\p\frac{\d\lambda}{\lambda^2}
\end{equation}
defines an equivalent norm on $\mathscr A_{1/2,1}$\,. In particular, different choices of the functions $\eta$ and $\zeta$ lead to equivalent norms. In \eqref{nuci}, $\eta(\lambda A)$ is defined by through the functional calculus for the self-adjoint operator $\lambda A$\,.
\item[(ii)]
There is also an admissible norm on $\mathscr A_{1/2,1}$ in terms of the resolvent of $A$\,:
\begin{equation*}\label{barlad}
\begin{aligned}
\p\!v\!\p'_{1/2,1}\,:=\;\p\!v\!\p+\int_0^1\tau^{1/2}\big\Vert A({\sf Id}-i\tau A)^{-1}v\big\Vert\frac{d\tau}{\tau}\,.
\end{aligned}
\end{equation*}
\end{enumerate}

\begin{Corollary}\label{lipia}
Let $L\in\mathbb B(\H)$ an operator satisfying one of the two (equivalent) conditions:
\begin{enumerate} 
\item[(i)] For some choices $\eta,\zeta$ as above and for some positive constant $C$, one has
\begin{equation*}\label{caldarusani}
\int_0^1\!\p\!\eta(\lambda A)Lu\!\p\frac{\d\lambda}{\lambda^2}\le C\!\p\!u\!\p\,,\quad\forall u\in\H\,.
\end{equation*}
\item[(ii)]
\begin{equation*}\label{caldarushani}
\int_0^1\tau^{1/2}\big\Vert A({\sf Id}-i\tau A)^{-1}Lu\big\Vert\frac{d\tau}{\tau}\le C\!\p\!u\!\p\,,\quad\forall u\in\H\,.
\end{equation*}
\end{enumerate}
Then $L^*|H|^{1/2}$ is an $H$-smooth operator. 
\end{Corollary}

\begin{proof}
The assertion follows from Lemma \ref{silistea} (iii) and the two characterizations of the space $\mathscr A_{1/2,1}$ described above.
\end{proof}

A simplification (but also a weakening) is possible by abstract properties of interpolation spaces:

\begin{Corollary}\label{lippia}
The operator $\,({\sf Id}+|A|)^{-\th}|H|^{1/2}$ is $H$-smooth for every $\th>1/2$\,.
\end{Corollary}

\begin{proof}
We recall the continuous embeddings of real interpolation spaces
\begin{equation*}\label{pucheni}
\mathscr A_{s,p}\hookrightarrow\mathscr A_{t,q}\quad{\rm if}\ s>t\,,\ p,q\in[1,\infty]\,.
\end{equation*}
In particular, we need the case $\,\mathscr A_{\th,2}\hookrightarrow\mathscr A_{1/2,1}$ where $\th>1/2$\,. Since $\mathscr A$ is the domain of the self-adjoint operator $A$ in the Hilbert space $\H$\,, then $\,\mathscr A_{\th,2}$ coincides with $\mathscr A_{\th}:={\sf Dom}\big(|A|^\th\big)$ (which can also be defined by complex interpolation) and the norm in $\,\mathscr A_{\th,2}$ is equivalent with the graph norm of $|A|^\th$. Thus
\begin{equation*}\label{scrovistea}
({\sf Id}+|A|)^{-\th}\in\mathbb B\big(\H,\mathscr A_{\th}\big)\subset\mathbb B\big(\H,\mathscr A_{1/2,1}\big)
\end{equation*}
and then we apply again Lemma \ref{silistea} (iii).
\end{proof}

\section{Operators on graded Lie groups}\label{fralticeni}

The results of the previous sections can be applied to operators on {\it homogeneous Lie groups}, since (by definition) these are endowed with a family of dilations leading to an interesting unitary group, crucial in the development of their theory. We recall that homogeneous Lie groups are connected, simply connected and nilpotent, and they almost coincide with this class of groups. Since there seems to be little more to say than just copying the results above in this context, we prefer to turn to the smaller class of {\it graded groups}, which allow more concrete statements. In this section we are going to review briefly some basic facts. Much more information can be found in \cite[Ch.\;4]{FR}.

\smallskip
Let $\G$ be a graded Lie group of step $r$, with unit $\e$ and Haar measure $dx$\,. Its Lie algebra can be written as a direct sum of vector subspaces
\begin{equation*}\label{lie}
\g=\mathfrak v_1\oplus\dots\oplus \mathfrak v_r\,,
\end{equation*}
where $\,[\mathfrak v_k,\mathfrak v_l]\subset\mathfrak v_{k+l}$ for every $k,l\in\{1,\dots,r\}$\,; if $\,k+l> r$ we just set $\mathfrak v_{k+l}=\{0\}$\,. Then $\G$ is a connected and simply connected nilpotent Lie group and thus the exponential map ${\sf exp}:\mathfrak g\to\G$ is a diffeomorphism with inverse denoted by ${\sf log}$\,. Let us set
\begin{equation*}\label{dim}
N_k:=\dim\mathfrak v_k\,,\quad N:=\dim\g=N_1+N_2+\dots+N_r\,,
\end{equation*}
and define {\it the homogeneous dimension} 
\begin{equation*}\label{homdim}
M:=N_1+2N_2+\dots+rN_r\,.
\end{equation*}

We are going to use basis $\{X_1,\dots,X_N\}$ of $\g$ such that for every $k$ the $N_k$ vectors $\big\{X_{j}\mid N_1+\dots+N_{k-1}<j\le N_1+\dots+N_{k-1}+N_k\big\}$ generate $\mathfrak v_k$  (we set $N_0=0$ for convenience)\,. For any $x\in\G$ one decomposes ${\sf log}\,x=\sum_{j=1}^N{\sf x}_jX_j$\,, which defines coordinates on the group $\G\ni x\to({\sf x}_1,\dots,{\sf x}_N)\in\R^N$. 

\smallskip
Usually one introduces the dilations on the graded group $\G$ in terms of the multiplicative group $(\R_+,\cdot)$\,. Making use of the group isomorphism $(\R,+)\ni t\to e^t\in(\R_+,\cdot)$\,, we switch to an action of the additive group $\R$ by automorphisms of the Lie algebra $\g=\mathfrak v_1\oplus\dots\oplus \mathfrak v_r$ given by
\begin{equation*}\label{dila}
\mathfrak{dil}_t\big(Y_{(1)},Y_{(2)},\dots,Y_{(r)}\big):=\big(e^t Y_{(1)},e^{2t}Y_{(2)}\dots,e^{rt}Y_{(r)}\big)\,,\quad t\in\R\,,\ Y_{(k)}\in\mathfrak v_k\,,\ 1\le k\le r\,.
\end{equation*}
One has 
\begin{equation}\label{stoarpha}
\mathfrak{dil}_t(X_j)=e^{\nu_j t} X_j\,,\quad t\in\R\,,\ 1\le j\le N,
\end{equation}
in terms of the {\it dilation weights} $\,\nu_j:=k$ if $\,N_1+\dots+N_{k-1}< j\le N_1+\dots+N_k$\,. Then we transfer the dilations to the group by 
\begin{equation}\label{ansfer}
{\sf dil}_t(x):={\sf exp}\big[\mathfrak{dil}_t({\sf log}\,x)\big]\,,\quad x\in\G\,,\ t\in\R\,.
\end{equation}
This induces a unitary strongly continuous $1$-parameter group on $\H:=L^2(\G)$ by
\begin{equation}\label{sacele}
\big[{\sf Dil}(t)u\big](x):=e^{\frac{Mt}{2}}\big(u\circ{\sf dil}_t\big)(x)=e^{\frac{Mt}{2}}u\big({\sf dil}_t(x)\big)\,.
\end{equation}
This will be the group $W$ involved in the abstract constructions and results of sections \ref{firtanunusus} and \ref{firtanunus}. In terms of the coordinate functions defined before, its infinitesimal generator, uniquely defined by ${\sf Dil}(t)=e^{itA}$, is given by
\begin{equation*}\label{generatoru}
A=\frac{1}{i}\sum_{j=1}^N \nu_j{\sf x}_j X_j+\frac{M}{2i}{\rm Id}=\frac{1}{i}\sum_{j=1}^N \nu_j X_j{\sf x}_j-\frac{M}{2i}{\rm Id}\,.
\end{equation*}
The operator $A$ depends on the homogeneous structure of $\G$\,, but not on the chosen basis.

\smallskip
We discuss now {\it homogeneous Rockland operators}, both because they provide operators $H$ to be studied and because they are useful to express $H$-smoothness. By definition, a Rockland operator is a (say left) invariant differential operator $R$ on $\G$ such that, for every non-trivial irreducible representation $\pi:\G\to\mathbb B(\H_\pi)$\,, the operator $d\pi(R)$  is injective on the subspace $\H_\pi^\infty$ of all smooth vectors of the representation. It exists when the group $\G$ is graded, with positive integer weights $\nu_j$\,. It is shown that such operators also exist subject to the requirement to be homogeneous and positive: the homogeneity reads, using \eqref{sacele}
\begin{equation*}\label{brasov}
{\sf Dil}(-t)R\,{\sf Dil}(t)=e^{q t}R\,,\quad\forall\,t\in\R\,.
\end{equation*}
The degree of homogeneity $q$ is a multiple of any of the dilation weights. An important fact (that we will not use) is that a left invariant homogeneous differential operator is hypoelliptic if and only if it is a Rockland operator \cite{HN,FR}.

\smallskip
We also assume that $R$ is positive; by \cite[Cor.\! 4.3.4]{FR} it is also injective, so $0$ is not an eigenvalue. As a consequence, it satisfies the strong positivity property
\begin{equation*}\label{ploiesti}
\<Ru,u\>>0\,,\quad\forall\,u\in{\sf Dom}(R)\setminus\{0\}\,.
\end{equation*}

\begin{Example}\label{timisjos}
{\rm If $l\in\N_0$\,, the power $R^l$ of a $q$-homogeneous positive Rockland operator is a $ql$-homogeneous positive Rockland operator. The transpose $H^\dag$ has the same properties as $R$\,. Concrete positive examples of homogeneous degree $q=2p$ are 
\begin{equation*}\label{ghimbav}
R:=\sum_{j=1}^{N'}(-1)^{\frac{p}{\nu_j}}Z_j^{2\frac{p}{\nu_j}},
\end{equation*}
where $\{Z_j\}_{j=1,\dots,N}$ is a basis as in \cite[Lemma 3.1.14]{FR} and $p$ is a commun multiple of the dilation weights. The basis is such that $Z_j$ is $\nu_j$-homogeneous, $Z_1,\dots,Z_{N'}$ generates $\g$ as a Lie algebra, while $Z_{N'+1},\dots,Z_N$ generates a vector space containing $[\g,\g]$\,.
}
\end{Example}

We are not going to review the homogeneous $\big\{\dot L^2_{\si}(\G)\!\mid\!\si\in\R\}$ and inhomogeneous $\big\{L^2_{\si}(\G)\!\mid\!\si\in\R\}$ Sobolev spaces on the graded group $\G$\,, since a clear and comprehensive presentation can be found in \cite[Sect.\;4.4]{FR}. They can also be deduced from the way we introduced the spaces $\mathcal G^\si$ and $\dot{\mathcal G}^\si$ at the beginning of Section \ref{firtanunusus}: one just has to replace $|H|$ with a suitable power of a positive homogeneous Rockland operator. We recall that we allow different equivalent scalar products on these Hilbertizable spaces.

\smallskip
The results of this article apply to {\it self-adjoint injective homogeneous operators on graded Lie groups}. It is not our intention to explore these conditions in a very general setting, this being an arduous task. We only give some simple examples and indicate informly others that could be in principle fully treated with extra work and/or imposing further requirements. We assume that the order $\beta$ of homogeneity is strictly positive, as we did before, but this in not important.
\begin{enumerate}
\item[(i)]
First of all, $H$ could be a multiplication operator $H=m(x)$ with some continuous (or measurable) function $m:\G\to\mathbb R$\,. Such an operator is self-adjoint on ${\sf Dom}(H):=\{u\in L^2(\G)\mid mu\in L^2(\G)\}$\,. The operator is homogeneous of order $\beta$ if and only if the function is homogeneous of order $-\beta$\,, i.e.\! if $m\big({\sf dil}_{-t}(x)\big)=e^{\beta t}m(x)$ for all $x,t$\,. It is injective if and only if $\{x\in\G\mid m(x)=0\}$ is negligible with respect to the Haar measure (and this is governed by the behavior of $m$ on "the unit sphere"). Clearly, $|H|^{1/2}$ is the operator of multiplication with the function $|m|^{1/2}$\,. Particular cases are $m(x):=p([x])$\,, where $[\cdot]$ is an homogeneous quasi-norm, $p:\mathbb R_+\to\R$ is homogeneous in the usual sense and $p(\lambda)\ne 0$ for some (and thus for all) $\lambda\ne 0$\,.
\item[(ii)]
Rockland operators yield particular cases of operators $H$ to which we can apply the already obtained results: If $R$ is a positive Rockland operator of homogeneous degree $q$ and $\mu>0$\,, then the fractional power $H:=R^\mu$ is self-adjoint, injective and homogeneous of order $\beta:=\mu q$\,. In terms of Soblolev spaces one has the identifications
\begin{equation*}\label{babadag}
\mathcal G^{\pm 2}=L^2_{\pm\beta}(\G)\,,\quad\mathcal G^{\pm 1}=L^2_{\pm\beta/2}(\G)\,,\quad\dot{\mathcal G}=\dot L^2_{\beta/2}(\G)\,,\quad\dot{\mathcal G}^*=\dot L^2_{-\beta/2}(\G)\,.
\end{equation*}
More generally, even if $R$ is not positive, $H\!:=p(R)$ is admissible if $p:\mathbb R_+\!\to\R$ is homogeneous and $p(\lambda)\ne 0$ if $\lambda\ne 0$\,.
In particular, by Remark \ref{filiasi}, we know that $p(R)$ has purely absolutely continuous spectrum.
\item[(iii)] 
In terms of our basis $\{X_1,\dots,X_N\}$ and a multi-index $\delta:=(\delta_1,\dots,\delta_N)$\,, the operator $X^\delta\!:=X_1^{\delta_1}\dots X_N^{\delta_N}$ is homogeneous of degree $\nu_1\delta_1+\dots+\nu_N\delta_N$\,. One can combine such differential operators with homogeneous multiplication operators from (i) and construct formally self-adjoint homogeneous operators. But rigorous self-adjointness and injectivity can be very hard in general.

\item[(iv)]
A large and interesting class of operators associated to a graded group have a pseudo-differential nature and is described in \cite[Sect.\;5]{FR}; see also references therein. The construction involves operator-valued symbols and  the irreducible representation theory of the group and is too complicated to be mentioned here, needing the introduction of many notions of harmonic analysis.
But in \cite[Sect.\;8]{MR} (see also \cite{Me,Gl1,Gl2} for the invariant case) another (but equivalent) type of pseudo-differential operators are built from scalar-valued symbols $a$ defined on the cotangent bundle $\,T^*Ò\G\cong\G\times\mathfrak g^*$, where $\g^*$ is the dual of the Lie algebra $\g$ of $\G$\,. The quantization rule makes use of the diffeomorphism $\log:\G\to\g$ and the duality 
$$
\g\times\g^*\ni(X,\mathcal X)\to\<X\mid\mathcal X\>:=\mathcal X(X)
$$ 
and is given by
$$
(Hu)(x)\equiv[\mathfrak{Op}(a)u](x):=\int_{\G}\int_{\g^*}e^{i\<\log(y^{-1}x)\,\mid\,\mathcal X\>}a(x,\mathcal X)u(y)dy d\mathcal X
$$
in terms of the Lebesgue measure $d\mathcal X$, eventually normalized accordingly to $\{\mathcal X_1,\dots,\mathcal X_N\}$, the dual basis of $\{X_1,\dots,X_N\}$\,. One computes formally
\begin{equation*}\label{voinesti}
[{\sf Dil}(-t)\mathfrak{Op}(a){\sf Dil}(t)u](x)=e^{-\frac{Mt}{2}}\!\int_{\G}\int_{\g^*}\!e^{i\<\log[{\sf dil}_{-t}(y^{-1}x)]\,\mid\,\mathcal X\>}a\big({\sf dil}_{-t}(x),\mathcal Y\big)u(y)dy d\mathcal Y
\end{equation*}
which, using the formula \eqref{ansfer} and the natural dual action by dilations on $\g^*$ can be written
\begin{equation*}\label{voinicesti}
[{\sf Dil}(-t)\mathfrak{Op}(a){\sf Dil}(t)u](x)=e^{-\frac{Mt}{2}}\!\int_{\G}\int_{\g^*}\!e^{i\<\log(y^{-1}x)\,\mid\,\mathfrak{dil}^*_{-t}(\mathcal Y)\>}a\big({\sf dil}_{-t}(x),\mathcal Y\big)u(y)dy d\mathcal Y
\end{equation*}
If $\{\mathcal X_1,\dots,\mathcal X_N\}$ is the dual basis of $\{X_1,\dots,X_N\}$\,, from \eqref{stoarpha} one gets immediatly for every $j$ that $\,\mathfrak{dil}^*_{-t}(\mathcal X_j)=e^{-\nu_j}\mathcal X_j$\,, and after a change of variables this leads to
\begin{equation*}\label{voynicesti}
{\sf Dil}(-t)\mathfrak{Op}(a){\sf Dil}(t)=\mathfrak{Op}(a_t)\,,\quad a_t(x,\mathcal X):=a\big({\sf dil}_{-t}(x),\mathfrak{dil}^*_{t}(\mathcal X)\big)\,,
\end{equation*}
so the homogeneity of the pseudo-differential operator can be extracted from that of the symbol. Note that, at least formally, the left-invariant operators are obtained letting $a$ be independent on $x$ and they are right convolution operators by a suitable inverse Fourier transform of $a$\,. Once again rigorous self-adjointness and injectivity are difficult to decide at such a general level.
\item[(v)]
In the commutative case $\G=\R^N\equiv\g$\,, the operator $A$ is just the generator of the usual dilations acting at the level of the group as $\,(x_1,\dots,x_N)\to(e^t x_1,\dots,e^t x_N)$\,. One has $M=N$ and $\nu_N=1$\,. In this case there is a rather general class of operators that can be treated with precision. Formally it could be obtained from (iv) making suitable identifications and choosing $a$ independent of the first variable. Independently and rigorously, one picks a measurable function $b:\R^n\to\R$ and define the self-adjoint operator $H:=b(D)$ by the functional calculus associated to the family $D:=(D_1=-i\partial_1,\dots,D_N=-i\partial_N)$ of $N$ commuting self-adjoint operators. It can also be seen as a convolution operator. Of course, if $b$ is a polynomial, the outcome is a differential operator with constant coefficients. The homogeneity of $b$ is responsable for the homogeneity of the corresponding operator. The injectivity is equivalent to $b^{-1}(\{0\})$ being Lebesgue-negligible, which, under the homogeneity assumption, is equivalent to $b^{-1}(\{0\})\cap S_{\R^N}$ being negligible in the unit sphere. This is much weaker than the usual ellipticity condition $b(\xi)\ne 0$ if $\xi\ne 0$\,, implying in particular that $b$ has a definite sign if $N\ge 2$\,. A symbol $b$ which is admissible for us could not be smooth. If it is, it could be dispersive ($(\nabla b)(\xi)\ne 0$ if $\xi\ne 0$) or not. 
\end{enumerate}

\section{Global smoothness for operators on graded Lie groups}\label{falticeni}

As said above, Rockland operators yield particular cases of operators $H$ to which we can apply the already obtained results. 
Besides this, another reason to mention Rockland operators is connected to the necessity to simplify results as Corollary \ref{lippia} and bring them to a form easier to compare with the results known for $\G=\R^N$. 
It is useful to fix a "first order" {\it defining operator $\mathbf D$ for Sobolev spaces}. For this we select (arbitrarily) a positive Rockland operator $R_0$ of homogeneous order $q$ and set $\mathbf D:=R_0^{1/q}$. For $\G=\R^N$ for example, one could think of the (positive) Laplacian $R_0=\Delta:=-\partial_1^2-\dots-\partial_N^2$ and then $\mathbf D=\Delta^{1/2}=|D|$\,, where $D:=-i\nabla$. The important property is that, for every $\si$\,, the following expressions define admissible norms on the Sobolev spaces (inhomogeneous and homogeneous, respectively)
\begin{equation*}\label{lehliu}
\p\!u\!\p_{L^2_\si(\G)}^2\,\sim\,\p\!u\!\p^2+\p\!\mathbf D^\si\!u\!\p^2\,,\quad \p\!u\!\p_{\dot L^2_\si(\G)}\,\sim\,\p\!\mathbf D^\si\!u\!\p\,.
\end{equation*}

The operator of multiplication by a Borel function $f:\G\to\mathbb C$ in $L^2(\G)$ will be denoted by $f(x)$\,. 
We are looking for $H$-smooth operators of the form $\,f(x)g(\mathbf D)|H|^{1/2}$, having Lemma \ref{silistea} in mind. Since $\mathbf D$ is positive and $0$ is not an eigenvalue, it will be enough to consider functions $g$ only defined on $(0,\infty)$\,. In the simplest case $f$ and $g$ are bounded function; but in the next section we will need the case in which $f(x)g(\mathbf D)$\,, a priori defined on a dense subset, extends to an element of $\mathbb B(\H)$ without each factor being bounded.

\smallskip
The next Lemma will be an important tool below. Recall that $\nu_N$ is the largest dilation weight.

\begin{Lemma}\label{mosneni}
Let $\,\Phi:(0,\infty)\to\mathbb C\setminus\{0\}$ be a $C^1$-function and $\,\Psi:\G\setminus\{\e\}\to\mathbb C$ a continuous function, such that the following functions are bounded for every $j=1,\dots,N$:
\begin{equation}\label{pascany}
\lambda\Phi'(\lambda)\Phi(\lambda)^{-1},\quad\lambda^{\nu_j}\Phi(\lambda)\,,\quad{\sf x}_j\Psi(x)\,.
\end{equation}
Assume in addition that $\Phi(\mathbf D)\Psi(x)\in\mathbb B(\H)$\,. Then the operator $A\Phi(\mathbf D)\Psi(x)$ is bounded .
\end{Lemma}

\begin{proof}
From the homogeneity of $\mathbf D$ with respect to $A$\,, written $e^{-itA}\mathbf D e^{itA}=e^t\mathbf D$\,, one deduces 
$$
e^{-itA}\Phi(\mathbf D)e^{itA}=\Phi(e^t\mathbf D)\,,\quad t\in\R\,.
$$ 
Taking the derivative in $t=0$ one gets 
\begin{equation*}\label{valcea}
[A,\Phi(\mathbf D)]=i\mathbf D\Phi'(\mathbf D)\,,
\end{equation*}
so we have to check that
\begin{equation}\label{busteni}
\Phi(\mathbf D)A\,\Psi(x)+i\mathbf D\Phi'(\mathbf D)\Psi(x)\in\mathbb B(\H)\,.
\end{equation}
The second term is easily seen to be bounded, writting it as 
$$
i\mathbf D\Phi'(\mathbf D)\Phi(\mathbf D)^{-1}\Phi(\mathbf D)\Psi(x)\,.
$$ 
Taking into account the explicit form of the operator $A$ and the fact that $\Phi(\mathbf D)\Psi(x)$ is bounded\,, the first one only requires
\begin{equation*}\label{poianatapului}
\Phi(\mathbf D)X_j{\sf x}_j\Psi(x)\!\in\mathbb B(\H)\,,\quad\forall\,j=1,\dots,N.
\end{equation*}
The operator $X_j$ is homogeneous of degree $\nu_j$\,. By \cite[Th.\,4.4.16]{FR} it maps continuously $\dot L^2_{s+\nu_j}(\G)$ to $\dot L^2_s(\G)$ for any $s\in\R$\,. Setting $s=-\nu_j$ and since $\mathbf D$ is defining for the homogeneous Sobolev spaces, it follows immediately that $\,\mathbf D^{-\nu_j}X_j$ is bounded. So, using \eqref{pascany} once again, we have
\begin{equation*}\label{radauti}
\Phi(\mathbf D)X_j{\sf x}_j\Psi(x)=\Phi(\mathbf D)\mathbf D^{\nu_j}\mathbf D^{-\nu_j}X_j{\sf x}_j\Psi(x)\in\mathbb B(\H)\,,
\end{equation*}
finishing the proof. 
\end{proof}

\begin{Remark}\label{calarasi}
{\rm It is easy to see from the proof that the norm $\p\!A\Phi(\mathbf D)\Psi(x)\!\p$ only depends (linearly or quadratically) on 
\begin{equation}\label{calafat}
\sup_{\lambda>0}|\lambda\Phi'(\lambda)\Phi(\lambda)^{-1}|\,,\ \ \sup_{\lambda>0}\lambda^{\nu_j}|\Phi(\lambda)|\,,\ \ \sup_{x\in\G}|{\sf x}_j\Psi(x)|\,,\quad\p\!\Phi(\mathbf D)\Psi(x)\!\p. 
\end{equation}
}
\end{Remark}

\begin{Proposition}\label{ziceni}
Let $H$ be self-adjoint, injective and homogeneous. Let $\,\varphi:(0,\infty)\to(0,\infty)$ and $\,\psi:\G\to(0,\infty)$ be functions such that  for $j=1,\dots,N$
\begin{equation}\label{pocoscany}
\sup_{\lambda>0}\big[\lambda|\varphi'(\lambda)|\varphi(\lambda)^{-1}\big]<\infty\,,\quad\sup_{\lambda>0}\big[\lambda^{\nu_j}\varphi(\lambda)\big]<\infty\,,\quad\sup_{x\in\G}|{\sf x}_j\psi(x)|<\infty\,.
\end{equation}
Also assume that $\varphi(\mathbf D)^\si\psi(x)^\si$, defined as a sesquilinear form on $C^\infty_0(\G)$\,, extends to an element of $\,\mathbb B(\H)$ for any $\si\in[0,1]$ and that for any $u,v\in C^\infty_0(\G)$ the family $\big\{\big\<\psi(x)^z u,\varphi(\mathbf D)^{\overline z}v\big\>\mid \re z\in(0,1)\big\}$ is holomorphic. Then for every $\th\in(1/2,1]$ the operator 
$$
T_\th(\varphi,\psi;H):=\psi(x)^{\th}\varphi(\mathbf D)^{\th}|H|^{1/2}
$$ 
is globally $H$-smooth.
\end{Proposition}

\begin{proof}
The adjoint of $\,T_\th\equiv T_\th(\varphi,\psi;H)$ is $\,T_\th^*=|H|^{1/2}L_\th$\,, where $L_\th:=\varphi(\mathbf D)^\th\psi(x)^{\th}$. By Lemma \ref{silistea} (iii), it is enough to show that $\,L_\th\in\mathbb B\big(\H,\mathscr A_{\th}\big)\subset\mathbb B\big(\H,\mathscr A_{1/2,1}\big)$\,. This will be done by complex interpolation of holomorphic families of operators, using the pairs $(\H,\H)$ and $(\H,\mathscr A)$ of Hilbert spaces.

\smallskip
Let us set $\,\Omega:=\{z=\si+i\tau\in\mathbb C\mid \tau\in[0,1]\}$ and define $L:\Omega\to\mathbb B(\H,\H)$ by
\begin{equation}\label{saftica}
L_z:=\varphi(\mathbf D)^{z}\psi(x)^{z}=\varphi(\mathbf D)^{i\tau}\varphi(\mathbf D)^{\si}\psi(x)^{\si}\psi(x)^{i\tau}.
\end{equation}
Recall that $\mathscr A_\th=[\H,\mathscr A]_\th$ is a complex interpolation space. By a deep result \cite[Sect.\;3]{CJ}, we get (an extension) $L_\th\in\mathbb B\big(\H,\mathscr A_{\th}\big)$ if we show that 
\begin{equation}\label{chitila}
\p\!L_{i\tau}\!\p_{\mathbb B(\H,\H)}\,,\quad\p\!L_{1+i\tau}\!\p_{\mathbb B(\H,\mathscr A)}\,\sim\,\p\!L_{1+i\tau}\!\p_{\mathbb B(\H)}+\p\!AL_{1+i\tau}\!\p_{\mathbb B(\H)}
\end{equation}
are well-defined and have (say) at most an exponential growth in $\tau$. A more precise and generous condition, involving the Poisson kernel on the strip $\O$ \cite[pag\;93]{BL}, is available in \cite[Th.\;2]{CJ}. The fact that $\H=L^2(\G)$ is a separable Hilbert space is also important, in particular allowing to identify the two different types of complex interpolation spaces used in \cite{CJ}.

\smallskip
It is clear that the first norm in \eqref{chitila} is equal to $1$\,. In the second norm, the first term is bounded uniformly in $\tau$ as we said before. For the second term we apply our assumptions, Lemma \ref{mosneni} and Remark \ref{calarasi} with $\Phi:=\varphi^{1+i\tau}=\varphi\,\varphi^{i\tau}$ and $\Psi:=\psi^{1+i\tau}=\psi\,\psi^{i\tau}$. In \eqref{calafat} the second, the third and the forth expressions are finite and $\tau$-independent, since they do  not depend on the factors $\varphi^{i\tau}$ and $\psi^{i\tau}$. For the first one we can write
$$
\begin{aligned}
\sup_{\lambda>0}\lambda\Big\vert\big[\varphi(\lambda)^{1+i\tau}\big]'\varphi(\lambda)^{-1-i\tau}\Big\vert&=\sup_{\lambda>0}\lambda\Big\vert\big[(1+i\tau)\varphi'(\lambda)\varphi(\lambda)^{i\tau}\varphi(\lambda)^{-1-i\tau}\Big\vert\\
&=(1+\tau^2)^{1/2}\sup_{\lambda>0}\lambda\big\vert\varphi'(\lambda)\big\vert\varphi(\lambda)^{-1}
\end{aligned}
$$ 
and this is largely enough for applying \cite[Th.\;2]{CJ}.
\end{proof}

We recall that on any graded Lie group there exist homogeneous quasi-norms. We fix one of them $[\cdot]:\G\to[0,\infty)$\,; by definition, it is continuous and one has
$$
[x]=0\ \Leftrightarrow\ x=0\,,\quad[{\sf dil}_t(x)]=e^t\,[x]\,,\quad[x^{-1}]=[x]\,,\quad\forall\,x\in\G\,.
$$
One could take
$$
[x]:=|{\sf x}_1|^{1/\nu_1}+\dots+|{\sf x}_N|^{1/\nu_N}
$$
for example. The precise choice will not be important below, and this is not unexpected: on one hand, two homogeneous quasi-norms are always equivalent and on the other hand we did not specify the value of the absolute constants in the $H$-smoothness estimates. Besides $[\cdot]$\,, we can also use one of the usual norms on $\g$ transported by $\log$ on $\G$\,, as
$$
|x|:=|{\sf x}_1|+\dots+|{\sf x}_N|\,.
$$

\begin{Corollary}\label{urzziceni} 
Let $H$ be a self-adjoint, injective and homogeneous operator on the graded group $\G$ and let $\th\in(1/2,1]$\,. 
The operators 
$$
S_\th(H):=(1+[x])^{-\th\nu_N}({\sf Id}+\mathbf D)^{-\th\nu_N}|H|^{1/2}
$$
and
$$
S'_\th(H):=(1+|x|)^{-\th}({\sf Id}+\mathbf D)^{-\th\nu_N}|H|^{1/2}
$$
are globally $H$-smooth.
\end{Corollary}

\begin{proof}
The coordinate function ${\sf x}_j$ is homogeneous of degree $\nu_j\le\nu_N$. 
Then it is easy to check that the function $\,\varphi(\lambda):=(1+\lambda)^{-\nu_N}$ and any of the functions $\,\psi(x):=(1+|x|)^{-1}$ or $\,\psi(x):=(1+[x])^{-\nu_N}$ satisfy the assumptions of Proposition \ref{ziceni}. 
\end{proof}

\begin{Remark}\label{azuga}
{\rm Suppose that our operator has the form $H=R^\mu$ with $\mu>0$ and $R$ a positive Rockland operator of homogeneous order $m$\,. Using $R$ for the Sobolev defining operator $\mathbf D=R^{1/m}$ and making some easy changes in the arguments above, one gets that 
$$
(1+[x])^{-\th\nu_N}({\sf Id}+R)^{-\th\nu_N/m}R^{\mu/2}\ \,{\rm is\ an}\ R^\mu{\rm -smooth\ operator.}
$$ 
The function $\R_+\ni s\to(1+s)^{-\th\nu_N/m}s^{\mu/2}$ is {\it unbounded} for some admissible $\th$ if and only if $m\mu>\nu_N$. Note that $m\mu$ is the homogeneous order of our operator $H$\,, while $\nu_N$ is intrinsically attached to the graded group. For $\R^N$ one has $\nu_N=1$ and for the Heisenberg group $\nu_N=2$\,.
}
\end{Remark}

\begin{Corollary}\label{frumusani}
Besides being homogeneous and injective, the operator $H$ is also supposed left-invariant. 
\begin{enumerate}
\item[(i)]
Then the operator $\,(1+[y^{-1}x])^{-\th\nu_N}({\sf Id}+\mathbf D)^{-\th\nu_N}|H|^{1/2}$ is globally $H$-smooth for each $\th\in(1/2,1]$\,, uniformly in $y\in\G$\,.
\item[(ii)]
Let $\mu$ be a finite positive measure on $\G$\,.
For $\th\in(1/2,1]$\,, let us set 
\begin{equation}\label{bacau}
g_{\th,\mu}(x):=\int_\G\,(1+[y^{-1}x])^{-\th\nu_N}d\mu(y)\,.
\end{equation}
\end{enumerate}
Then the operator $\,g_{\th,\mu}(x)({\sf Id}+\mathbf D)^{-\th\nu_N}|H|^{1/2}$ is globally $H$-smooth.
\end{Corollary}

\begin{proof}
Let us denote by $\Lambda_y$ the operator of left translation by $y\in\G$\,, acting as $\big[\Lambda_y(u)](x):=u(y^{-1}x)$\,; it defines a unitary operator in $\H:=L^2(\G)$\,. Since $\mathbf D$ is a power of a Rockland operator, which is left-invariant, $({\sf Id}+\mathbf D)^{-\th\nu_N}$ commutes with $\Lambda_y$\,, while $|H|^{1/2}$ and $e^{itH}$ also commute with $\Lambda_y$\,, by our assumption. On the other hand, it is easy to check that 
$$
\Lambda_y(1+[x])^{-\th\nu_N}=(1+[y^{-1}x])^{-\th\nu_N}\Lambda_y
$$ 
(both sides involving multiplication operators in the variable $x$)\,. Using Corollary \ref{urzziceni}, the unitarity of left translations and setting $B(t):=({\sf Id}+\mathbf D)^{-\th\nu_N}|H|^{1/2}e^{itH}$, one can write
\begin{equation*}\label{eforienord}
\begin{aligned}
\int_\R\,\big\Vert (1+[y^{-1}x])^{-\th\nu_N}B(t)u\,\big\Vert^2 dt&=\int_\R\,\big\Vert \Lambda_y(1+[x])^{-\th\nu_N}\Lambda_y^{-1}B(t)u\,\big\Vert^2 dt\\
&=\int_\R\,\big\Vert (1+[x])^{-\th\nu_N}B(t)\Lambda_y^{-1}u\,\big\Vert^2 dt\\
&\le C\!\p\!\Lambda_y^{-1} u\!\p^2\,= C\!\p\!u\!\p^2,
\end{aligned}
\end{equation*}
showing $H$-smoothness with the same constant $C$. Then, using this and the Cauchy-Schwartz inequality,
\begin{equation*}\label{eforiesud}
\begin{aligned}
\int_\R\,\Big\Vert \int_\G\,(1+[y^{-1}x])^{-\th\nu_N}&d\mu(y)\,B(t)u\,\Big\Vert^2 dt\le\int_\R\Big[\int_\G\,\big\Vert(1+[y^{-1}x])^{-\th\nu_N} B(t)u\,\big\Vert \,d\mu(y)\Big]^2dt\\
&\le\int_\R\Big[\Big(\int_\G d\mu(y)\Big)^{1/2} \Big(\int_\G\big\Vert(1+[y^{-1}x])^{-\th\nu_N} B(t)u\,\big\Vert^2 d\mu(y)\Big)^{1/2}\,\Big]^2dt\\
&=\int_\G d\mu(y)\int_\R\int_\G \big\Vert(1+[y^{-1}x])^{-\th\nu_N} B(t)u\,\big\Vert^2 d\mu(y)dt\\
&=\Big(\int_\G d\mu(y)\Big)^2\int_\R\big\Vert(1+[y^{-1}x])^{-\th\nu_N} B(t)u\,\big\Vert^2 dt\\
&\le\,\mu(\G)^2\,C\!\p\!u\!\p^2.
\end{aligned}
\end{equation*}
\end{proof}

\begin{Remark}\label{uranus}
{\rm Setting $h_\th(x):=(1+[x])^{-\th\nu_N}$, one checks by a short calculation in polar coordinates that $h_\th\in L^p(\G)$ if and only if $\,\th\nu_N p>M$. Then, by Young's inequality, 
$$
g_{\th,\mu}=\mu\ast h_\th\in \underset{p>M/\th\nu_N}{\cap}L^p(\G)\,.
$$ 
}
\end{Remark}

\section{Global smoothness for operators on stratified groups}\label{stupoare}

We assume now that the group $\G$ is stratified: in addition to being graded, this means that $\mathfrak v_1$ generates $\g$ as a Lie algebra. The first $N_1:={\rm dim}(\mathfrak v_1)$ vectors of the basis will be chosen in $\mathfrak v_1$\,.  
Looking at the elements of the Lie algebra as derivations in spaces of functions on $\G$\,, one has {\it the horizontal gradient} $\,\nabla:=(X_1,\dots,X_{N_1})$ and {\it the horizontal sublaplacian}
\begin{equation}\label{subL}
\Delta:=-\sum_{j=1}^{N_1}X_j^* X_j=\sum_{j=1}^{N_1}X_j^2.
\end{equation} 
Note our {\it sign convention}: most often the operator defined in \eqref{subL} is denoted by $-\Delta$\,. It is known that (our) $\Delta$ is a self-adjoint positive operator in $L^2(\G)$\,; its domain is the Sobolev space $L^2_2(\G)$\,. It can be seen as a Rockland operator \cite[Lemma 4.1.7]{FR}. Then $\,\Delta^{\alpha}$ is also a self-adjoint positive operator; its domain is the Sobolev space $L^2_{2\alpha}(\G)$ and it is homogeneous of degree $\beta=2\alpha$\,.
It is clear that a defining operator $\mathbf D$ for the Sobolev spaces is $|D|:=\Delta^{1/2}$. The notation $|D|$ is convenient having the case $\G=\R^N$ in view; we do not claim anything special about it.

\smallskip
First we write down reformulations of Corollary \ref{urzziceni} for this particular case:

\begin{Corollary}\label{urtziceni}
Assume that $H$ is a self-adjoint injective homogeneous operator on the stratified group $\G$\,. For every $\th\in(1/2,1]$ the operators
\begin{equation*}\label{trincabesti}
S_\th(H):=(1+[x])^{-\th\nu_N}\big({\sf Id}+|D|\big)^{-\th\nu_N}|H|^{1/2},
\end{equation*}
\begin{equation*}\label{trincabulesti}
\tilde S_\th(H):=(1+[x])^{-\th\nu_N}({\sf Id}+\Delta)^{-\th\nu_N/2}|H|^{1/2}
\end{equation*} 
are globally $H$-smooth. In particular, this applies to $H=\Delta^\alpha$ for every $\alpha>0$\,.
\end{Corollary}

\begin{proof}
The first one is obtained just by particularization. The second is left as an exercice.
\end{proof}

We work now on an improvement, based on a Hardy-type inequality and the fact that in Proposition \ref{ziceni} we did not ask the functions $\varphi,\psi$ to be bounded. In the next result one supposes that $\,\nu_1=1<M/2$\,; for most of the stratified groups (but not for the commutative cases $\R$ and $\R^2$) this holds automatically.

\begin{Proposition}\label{giurgiu}
Suppose that $\G$ is a stratified group with a homogeneous quasi-norm $[\cdot]$ and $H$ is a self-adjoint injective homogeneous operator in $L^2(\G)$\,. Assume that $M\ge 3$\,. For any $\th\in(1/2,1]$ the operator
\begin{equation*}\label{hagrig}
U_{\th}(H):=(1+[x])^{-\th(\nu_N-1)}[x]^{-\th}({\sf Id}+|D|)^{-\th(\nu_N-1)}|D|^{-\th}|H|^{1/2}
\end{equation*}
is globally $H$-smooth.
\end{Proposition}

\begin{proof}
We have to show that the pair of functions
\begin{equation*}\label{strigati}
\varphi(\lambda):=\lambda^{-1}(1+\lambda)^{-\nu_N+1}\,,\quad\psi(x):=[x]^{-1}(1+[x])^{-\nu_N+1}
\end{equation*}
satisfies the conditions of Proposition \ref{ziceni}.

\smallskip
Computing, one gets
$$
\lambda\varphi'(\lambda)=\Big[-1+(1-\nu_N)\frac{\lambda}{1+\lambda}\Big]\varphi(\lambda)\,,
$$
so the first condition in \eqref{pocoscany} is verified. One has 
$$
\lambda^{\nu_j}\varphi(\lambda)=\lambda^{\nu_j-1}(1+\lambda)^{-\nu_N+1}=\big[\lambda^{\nu_j-1}(1+\lambda)^{-\nu_j+1}\big](1+\lambda)^{\nu_j-\nu_N}
$$ 
and both factors are bounded, because $1=\nu_1\le\nu_j\le\nu_N$ for every $j$\,. Similarly
$$
{\sf x}_j\psi(x)={\sf x}_j[x]^{-1}(1+[x])^{-\nu_N+1}={\sf x}_j[x]^{-\nu_j}[x]^{\nu_j-1}(1+[x])^{-\nu_N+1}
$$
is bounded, since the coordinate ${\sf x}_j$ is homogeneous of degree $\nu_j$ and $[\cdot]$ is homogeneous of degree $1$\,.

\smallskip
Recall that, by a Hardy estimate \cite[Sect.\;3]{CCR}, if $\,\gamma\in[0,M/2)$ one has 
\begin{equation}\label{predeal}
\Delta^{-\gamma/2}\,[x]^{-\gamma}=|D|^{-\gamma}\,[x]^{-\gamma}\in\mathbb B(\H)\,.
\end{equation}
So, for $\si\in[0,1]\subset[0,M/2)$ one has
$$
\varphi(|D|)^\si\psi(x)^\si=({\sf Id}+|D|)^{-\si(\nu_N-1)}|D|^{-\si}[x]^{-\si}(1+[x])^{-\si(\nu_N-1)}\in\mathbb B(\H)
$$ 
because of \eqref{predeal}. For $u,v\in C_0(\G)$\,, the holomorphy of the map $z\to\<\varphi(|D|)^z\psi(x)^z u,v\>$ in the domain $\{\re z\in(0,1)\}$ follows from \cite[Sect.\,3]{CCR}. Applying Proposition \ref{ziceni} finishes the proof.
\end{proof}

\begin{Example}\label{medgidia}
{\rm For the Heisenberg group $\G:={\sf H}_{2n+1}\ni(a_1\dots,a_n;b_1,\dots,b_n;s)$ one just has to put $N=2n+1$\,, $M:=2n+2$ and $\nu_N=2$\,. The canonical (positive) sublaplacian is 
$$
\Delta:=-\sum_{k=1}^n\Big(\partial_{a_k}-\frac{b_k}{2}\partial_s\Big)^2-\sum_{k=1}^n\Big(\partial_{b_k}+\frac{a_k}{2}\partial_s\Big)^2
$$ 
and for $[\cdot]$ one can choose the Koranyi homogeneous quasi-norm
$$
[(a,b,t)]:=\Big(\big[\sum_{k=1}^n(a_k^2+b_k^2)\big]^2+s^2\Big)^{1/4}.
$$
}
\end{Example}

\begin{Example}\label{midia}
{\rm If $\G=\R^N$ is Abelian, then $M=N$ and $\nu_1=\nu_N=1$\,. So, if $H$ is self-adjoint, injective and homogeneous, then 
$$
U_\th(H):=|x|^{-\th}|D|^{-\th}|H|^{1/2}
$$ 
is globally $H$-smooth for every $\th\in(1/2,1]$ and for the usual absolute value $|\cdot|$\,. If $\,H:=\Delta^\alpha=|D|^{2\alpha}$ we get that $U_\th(\Delta^\alpha):=|x|^{-\th}|D|^{\alpha-\th}$ is $\Delta^\alpha$-smooth if $\th\in(1/2,1]$\,. The case $\alpha=1$ has been obtained in \cite{KY}. We refer to \cite{Su,RSug2} for extensions.
}
\end{Example}


\end{document}